\documentclass[11pt,,a4paper]{article} 
\author{Dimitri Chikhladze\footnote{The author gratefully acknowledges support
of an International Macquarie University Scholarship.}}
\title{The Tannaka representation theorem for separable Frobenius
functors}
\date{}
\usepackage{amsfonts} 
\usepackage{amsmath}
\usepackage[all]{xy} 
\usepackage{amsthm}
\usepackage{graphicx}

\makeatletter

\newcommand{\pocorner}{\hbox to 10pt{{\vrule height10pt depth0pt width0.5pt}%
    \vbox to 10pt{{\hrule height0.5pt width9.5pt depth0pt}\vfill}}}
\newcommand{\@poexcursion}[1]{\save[]-<0pt+#1,-12pt>*{\pocorner}\restore}
\newcommand{\pbcorner}{\vbox to 0pt{\kern 5pt\hbox to 0pt{\kern 5pt%
      \vbox{{\hrule height0.5pt width9.5pt depth0pt}}%
      {\vrule height10pt depth0pt width0.5pt}\hss}\vss}}
\def\@pbexcursion[#1]{\save[]+DR-<16pt,-8pt>+<#1,0pt>*{\pbcorner}\restore}
\newcommand{\pbexcursion}{\@ifnextchar[\@pbexcursion{\@pbexcursion[0pt]}}

\def\arr@fn{\futurelet\arr@next}
\def\arr@dn{\def\arr@next}

\newtoks\arr@toks
\def\addtoarr@toks#1{\arr@toks=\expandafter{\the\arr@toks#1}}

\def\arr@upperlab{}
\def\arr@lowerlab{}
\def\arr@mods{}
\def\arr@uppermods{}
\def\arr@lowermods{}

\newbox\arr@box
\newdimen\arr@dimen
\newdimen\arr@spacer
\newif\ifarr@fixeddim

\def\arr@updatedimen#1{%
  \ifarr@fixeddim\else\setbox\arr@box=\hbox{$\m@th\scriptstyle #1$}%
  \ifdim\wd\arr@box>\arr@dimen \arr@dimen=\wd\arr@box\fi\fi}

\def\parsearr@@{%
    \ifx\space@\arr@next \expandafter\arr@dn\space{\arr@fn\parsearr@@}%
    \else\ifx ^\arr@next \arr@dn ^##1{\def\arr@upperlab{^{##1}}\arr@fn\parsearr@@}%
    \else\ifx _\arr@next \arr@dn _##1{\def\arr@lowerlab{_{##1}}\arr@fn\parsearr@@}%
    \else\ifx <\arr@next \arr@dn <##1>{\arr@fixeddimtrue\arr@dimen=##1%
      \arr@fn\parsearr@@}
    \else\addAT@\ifx\arr@next \addAT@\arr@dn[##1]{\def\arr@mods{##1}\arr@fn\parsearr@@}%
    \else\ifarr@fixeddim\else\arr@updatedimen{\arr@upperlab}%
        \arr@updatedimen{\arr@lowerlab}\advance\arr@dimen by \arr@spacer\fi%
      \arr@toks={\xy<0pt,0pt>*+{}\ar}%
      \expandafter\addtoarr@toks\expandafter{\arr@mods}%
      \expandafter\addtoarr@toks\expandafter{\arr@upperlab}%
      \expandafter\addtoarr@toks\expandafter{\arr@lowerlab}%
      \addtoarr@toks{<\arr@dimen,0pt>*+{}\endxy}%
      \arr@dn{\the\arr@toks}%
    \fi\fi\fi\fi\fi\arr@next}

\newcommand{\arrow}{%
  \arr@fixeddimfalse\arr@dimen=0pt\def\arr@upperlab{}\def\arr@lowerlab{}%
  \def\arr@mods{}\arr@fn\parsearr@@}

\def\parsearrpair@@{%
    \ifx\space@\arr@next \expandafter\arr@dn\space{\arr@fn\parsearrpair@@}%
    \else\ifx ^\arr@next \arr@dn ^##1{\def\arr@upperlab{^{##1}}\arr@fn\parsearrpair@@}%
    \else\ifx _\arr@next \arr@dn _##1{\def\arr@lowerlab{_{##1}}\arr@fn\parsearrpair@@}%
    \else\ifx <\arr@next \arr@dn <##1>{\arr@fixeddimtrue\arr@dimen=##1%
      \arr@fn\parsearrpair@@}
    \else\addAT@\ifx\arr@next \addAT@\arr@dn[{\arr@fn\parsearrpair@@@}%
    \else\ifarr@fixeddim\else\arr@updatedimen{\arr@upperlab}%
        \arr@updatedimen{\arr@lowerlab}\advance\arr@dimen by \arr@spacer\fi%
      \arr@toks={\xy<0pt,0pt>*+{}="a", <\arr@dimen,0pt>*+{}="b"}%
      \expandafter\addtoarr@toks\expandafter{\addAT@\ar<2pt>}%
      \expandafter\addtoarr@toks\expandafter{\arr@uppermods}%
      \addtoarr@toks{"a";"b"}%
      \expandafter\addtoarr@toks\expandafter{\arr@upperlab}%
      \expandafter\addtoarr@toks\expandafter{\addAT@\ar<-2pt>}%
      \expandafter\addtoarr@toks\expandafter{\arr@lowermods}%
      \addtoarr@toks{"a";"b"}%
      \expandafter\addtoarr@toks\expandafter{\arr@lowerlab}%
      \addtoarr@toks{\endxy}\arr@dn{\the\arr@toks}%
    \fi\fi\fi\fi\fi\arr@next}

\def\parsearrpair@@@{%
  \ifx u\arr@next \arr@dn u##1]{\def\arr@uppermods{##1}\arr@fn\parsearrpair@@}%
  \else\arr@dn l##1]{\def\arr@lowermods{##1}\arr@fn\parsearrpair@@}%
  \fi\arr@next}

\newcommand{\arrowpair}{%
  \arr@fixeddimfalse\arr@dimen=0pt\def\arr@upperlab{}\def\arr@lowerlab{}%
  \def\arr@uppermods{}\def\arr@lowermods{}\arr@fn\parsearrpair@@}

\makeatother

\newdir{ >}{{}*!/-10pt/@{>}}
\newdir{u(}{{}*!/-6pt/@^{(}}
\newdir{d(}{{}*!/-6pt/@_{(}}
\newdir{|>}{%
  !/4.5pt/@{|}*:(1,-.2)@^{>}*:(1,+.2)@_{>}*+@{}}

\newcommand{\epi}{\arrow@[@{->>}]}
\newcommand{\cover}{\arrow@[@{-|>}]}
\newcommand{\spanarr}{\arrow@[|-@{|}]}
\newcommand{\inc}{\arrow@[@{u(->}]}
\newcommand{\mapto}{\arrow@[@{|->}]}
\newcommand{\mono}{\arrow@[@{ >->}]}
\newcommand{\mat}{\arrow@[|-*{\object@{o}}]}
\newcommand{\twocell}{\arrow@[@{=>}]}

\input{diagxy}
\begin{document}
\newtheorem{theorem}{Theorem}
\newtheorem{lemma}[theorem]{Lemma}
\newtheorem{proposition}[theorem]{Proposition}
\newtheorem{corollary}[theorem]{Corollary}
\theoremstyle{definition}
\newtheorem{definition}[theorem]{Definition}

\maketitle

\begin{abstract}
A weak bialgebra is known to be a special case of a bialgebroid. In this paper
we study the relationship of this fact with the Tannaka theory of bialgebroids
as developed in \cite{P}. We obtain a Tannaka
representation theorem with respect to a separable Frobenius fiber functor.
\end{abstract}

In \cite{P} Tannaka theory was developed with respect to a fiber functor with
values in the category of two sided $R$-bimodules. The reconstructed object
in that case is a bialgebroid. We show that a separable Frobenius monoidal
functor gives rise to a fiber functor of that form. The reconstructed object
then is a bialgebroid with a separable Frobenius base monoid, which is the same
as a weak bialgebra. This gives an alternative approach to \cite{M} where the
weak bialgebra was reconstructed directly from the separable Frobenius functor.

Suppose that $\mathcal{V}$ is a braided monoidal category. Let
$\mathrm{Mod}\mathcal{V}$ denote the bicategory of modules. Objects of $\mathrm{Mod}\mathcal{V}$ are
the monoids $S$, $R$ in $\mathcal{V}$. The homcategory
$\mathrm{Mod}\mathcal{V}(S, R)$ is the category of Eilenberg-Moore objects for
a pseudomonad $S\otimes-\otimes R$. A 1-cell from $S$ to $R$, depicted $S \spanarr R$, is a
module from $S$ to $R$. A 1-cell from $I$ to $R$ is a right $R$-module.

By $k$ we denote a commutative ring and by $k\operatorname{-Mod}$ we denote the
monoidal category of $k$-modules. If $\mathcal{V} = k\operatorname{-Mod}$
then the monoidal bicategory of modules is denoted simply by $\mathrm{Mod}$.
The objects of $\mathrm{Mod}$ are $k$-algebras and the 1-cells are the modules between $k$-algebras.

\begin{definition}
(see \cite{path}) A module $M: S \spanarr R$ is called Cauchy if it has a left
adjoint $M^\ast : R \spanarr S$ in $\mathrm{Mod}\mathcal{V}$.
\end{definition}

There exists a characterization of Cauchy modules for the case $\mathcal{V} =
k\operatorname{-Mod}$.

\begin{theorem}\label{t1}  
A module $M : S \spanarr R$ between $k$-algebras is Cauchy if and only if it
is finitely generated and projective as a right $R$-module.
\end{theorem}  

Suppose that $R$ is a $k$-algebra. Let $\mathcal{C}$ be a monoidal category and
$F: \mathcal{C} \to \mathrm{Mod}(R, R)$ be a strong monoidal functor with image
in the Cauchy modules. Let $F_r : \mathcal{C} \to \mathrm{Mod}(k, R)$ be the
functor obtained from $F$ by forgetting the left action. Let $(-)^\ast :
\mathrm{Mod}(k, R) \to \mathrm{Mod}(R, k)$ be the functor taking the left
adjoint. Consider the two variable functor

\begin{equation}\label{coend}
\mathcal{C}^o\times \mathcal{C} \to^{F_l\times F_l}
\mathrm{Mod}(k, R)\times\mathrm{Mod}(k, R) \to^{(-)^\ast\times 1}
\mathrm{Mod}(R, k)\times
\mathrm{Mod}(k, R) \to^{\otimes_k} \mathrm{Mod}(R, R)\text{.}
\end{equation}

\noindent Let $L$ denote the coend of this. It
was shown in \cite{P} that $L$ is a bialgebroid with base $R$. Moreover, for
$k$ a field, the following representation theorem was obtained.

\begin{theorem}\label{theorem} (Ph\`ung H\^o Hai \cite{P})
Let $\mathcal{C}$ be a locally finite $k$-linear abelian monoidal category
and let $F : \mathcal{C} \to \mathrm{Mod}(R, R)$ be a
faithful, exact, strong monoidal functor with image in the Cauchy modules. Then
there is a monoidal equivalence between $\mathcal{C}$ and the monoidal category
of comodules over $L$.
\end{theorem}

A $k$-linear abelian category is locally finite if each of its hom vector spaces
is finite dimensional and each object has the composition series of finite
length. For description of the monoidal category of comodules over a
bialgebroid the reader may refer to \cite{P}.

Suppose that $F : \mathcal{W} \to \mathcal{V}$ is a monoidal functor with
structure maps

\begin{equation}\label{phi}
\phi_{X, Y} : F(X)\otimes F(Y) \to F(X\otimes Y), \qquad
\phi_0 : I \to F(I)\text{.}
\end{equation}

\noindent $R = F(I)$ becomes a monoid in $\mathcal{V}$ with multiplication and
unit:

$$F(I)\otimes F(I) \to^{\phi_{I, I}} F(I\otimes I) \cong F(I)\text{,} \qquad I
\to^{\phi_0} F(I)\text{.}$$

\noindent The functor $F$ can be factored through the forgetful functor
$U : \mathrm{Mod}\mathcal{V}(R, R) \to \mathcal{V}$ as shown in the diagram (see
\cite{S}):
 
$$ 
\bfig
\Dtriangle[\mathcal{C}`\mathrm{Mod}\mathcal{V}(R, R)`\mathcal{V};F`\tilde{F}`U]
\efig
$$

\noindent $\tilde{F}$ takes an object
$X$ to $F(X)$, which is a module $R \spanarr R$ with left and right actions 

$$F(I)\otimes F(X) \to^{\phi_{I, X}} F(I\otimes X) \cong F(X)$$
$$F(X)\otimes F(I) \to^{\phi_{X, I}} F(X\otimes I) \cong F(X) \text{.}$$

\noindent Commutativity of the
diagram

\begin{equation}\label{XY}
F(X)\otimes F(I)\otimes F(Y) \two^{\phi_{X,
I}\otimes1}_{1\otimes\phi_{I, X}} F(X)\otimes F(Y) \to^{\phi_{X, Y}} F(X\otimes Y)
\end{equation}
 
\noindent induces a map $\pi_{X, Y} : F(X)\otimes_R F(Y) \to F(X\otimes Y)$.
The maps $\pi_{X, Y}$ make $\tilde{F}$ into a normal (meaning that
$\tilde{F}(I) = I$) monoidal functor.

A functor $F$ between monoidal categories is called
Frobenius monoidal \cite{S1}, \cite{DP} if it is equipped with a monoidal
structure \eqref{phi} and an opmonoidal structure
 
$$\psi_{X, Y} : F(X\otimes Y) \to F(X)\otimes F(Y), \qquad \psi_0 : F(I) \to
I\text{,}$$
 
\noindent which satisfy 

$$
\bfig
\square<1100,500>[F(X\otimes Y)\otimes F(Z)`F(X)\otimes F(Y)\otimes
F(Z)`F(X\otimes Y\otimes Z)`F(X)\otimes F(Y\otimes Z);\psi\otimes
1`\phi`1\otimes\phi`\psi] 
\square(0,-850)<1100,500>[F(X)\otimes F(Y\otimes Z)`F(X)\otimes F(Y)\otimes
F(Z)`F(X\otimes Y\otimes Z)`F(X\otimes Y)\otimes F(Z);1\otimes
\psi`phi`\phi\otimes 1`\psi]
\efig
$$

\noindent A Frobenius monoidal functor is called separable if:

$$\bfig
\Atriangle[F(X\otimes Y)`F(X)\otimes F(Y)`F(X\otimes Y);\psi`\phi`1]
\efig$$
 
\begin{proposition}
(see \cite{S}) If $F$ is a separable Frobenius monoidal functor then
$\tilde{F}$ is a strong monoidal functor.
\end{proposition}

\begin{proof}
The diagram \eqref{XY} is a split coequalizer by the morphisms 

$$F(X\otimes Y) \to^{\psi_{X, Y}} F(X)\otimes F(Y) \qquad F(X)\otimes F(Y)
\to^{\psi_{X, I}\otimes 1} F(X)\otimes F(I)\otimes F(Y)$$

\noindent It follows that the map $\pi_{X, Y}$ defining the monoidal structure
on $\tilde{F}$ is an isomorphism.
\end{proof}

A separable Frobenius monoid in $\mathcal{V}$ is an object $R$
equipped with a monoid structure and a comonoid structure which satisfy the
Frobenius condition:

\begin{center}
\includegraphics[scale=0.3]{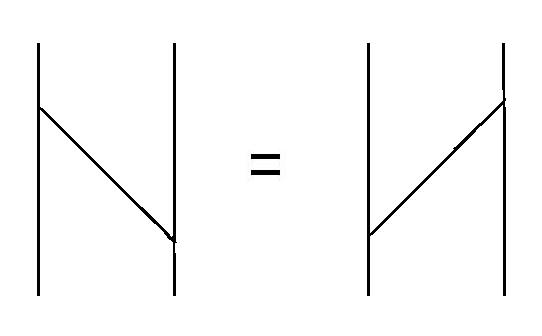}
\end{center}

\noindent and the separability condition:

\begin{center}
\includegraphics[scale=0.3]{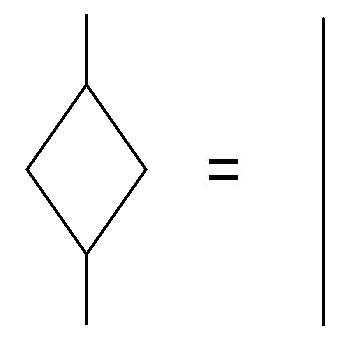}
\end{center}

If $F : \mathcal{W} \to \mathcal{V}$ is a separable Frobenius monoidal functor
then $R = F(I)$ is a separable Frobenius monoid in $\mathcal{V}$ with the monoid
structure coming from the monoidal structure of $F$ and the comonoid structure coming from the opmonoidal
structure on $F$.

\begin{proposition}\label{p}
Suppose that $R$ is a separable Frobenius monoid. A module $M : I \spanarr R$ is
Cauchy if the object $M$ has a left dual in $\mathcal{V}$.
\end{proposition}

\begin{proof}
Let $M^\ast$ be the left dual to $M$ in $\mathcal{V}$ with duality unit and
counit

\begin{center}
\includegraphics[scale=0.4]{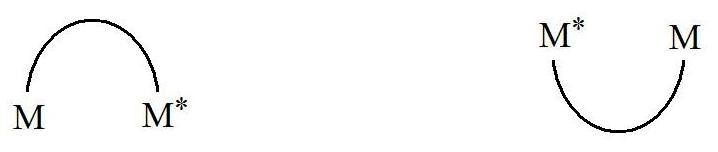}
\end{center}

\noindent $M^\ast$ becomes a module $M \spanarr I$ with action

\begin{center}
\includegraphics[scale=0.4]{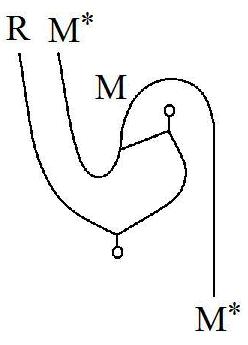}
\end{center}

\noindent Define a map $n : I \to M\otimes_R M^\ast$ as the composite of the
duality unit $I \to M\otimes M^\ast$ with the canonical injection $M\otimes
M^\ast \to M\otimes_R M^\ast$. Define a map $e : M^\ast\otimes M \to R$ by

\begin{center} 
\includegraphics[scale=0.4]{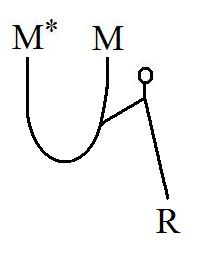}
\end{center}

\noindent The map $e$ is a map in $\mathrm{Mod}\mathcal{V}(R, R)$. The
following calculations show that it respects left and right $R$-actions.

\begin{center}
\includegraphics[scale=0.4]{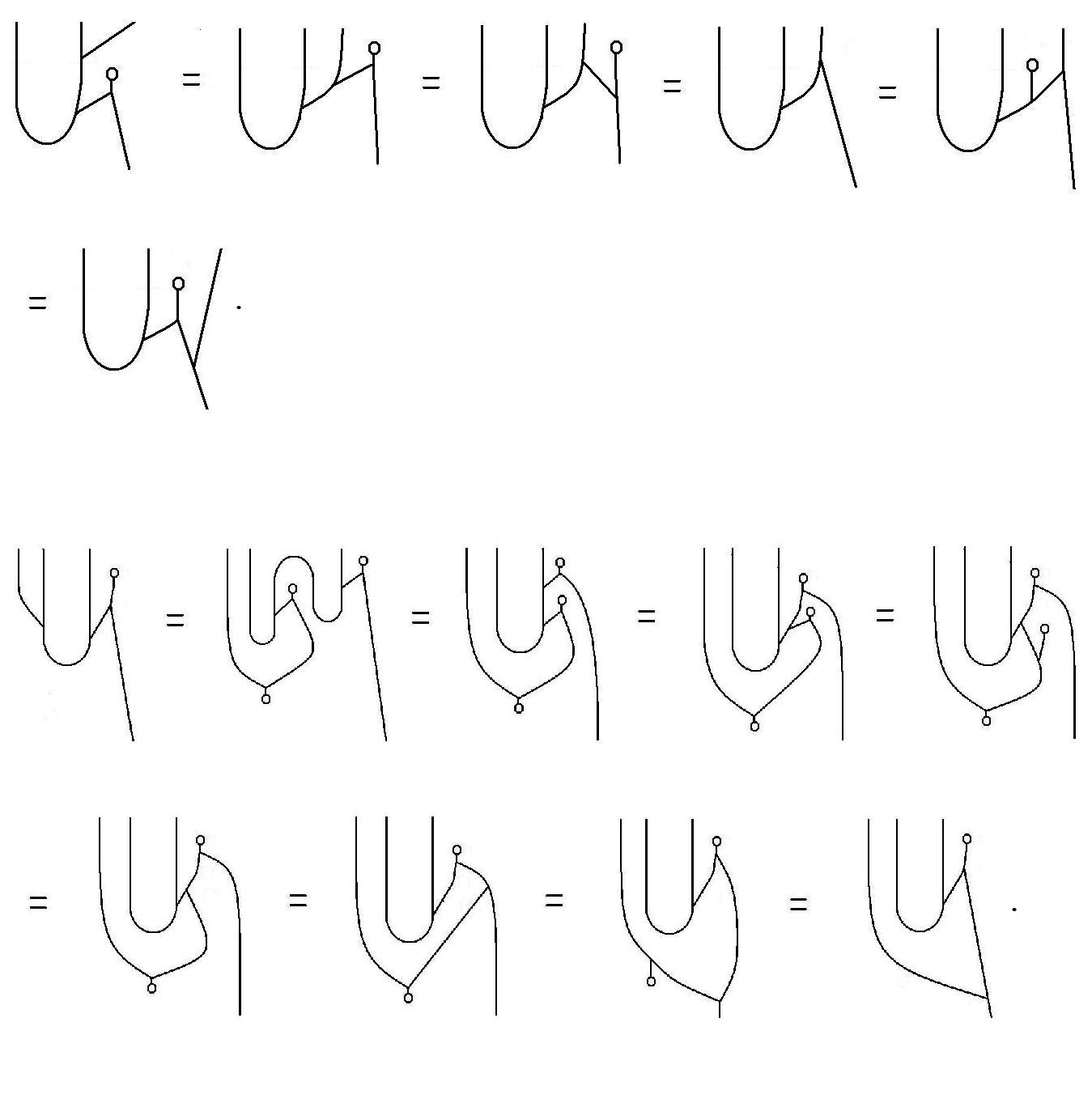}
\end{center}

$M^\ast : R \spanarr I$ is a left adjoint to $M : I \spanarr R$ in
$\mathrm{Mod}\mathcal{V}$ with unit $n$ and unit $e$. The triangle (or ``snake'')
identities are established by the following calculations.

\begin{center}
\includegraphics[scale=0.4]{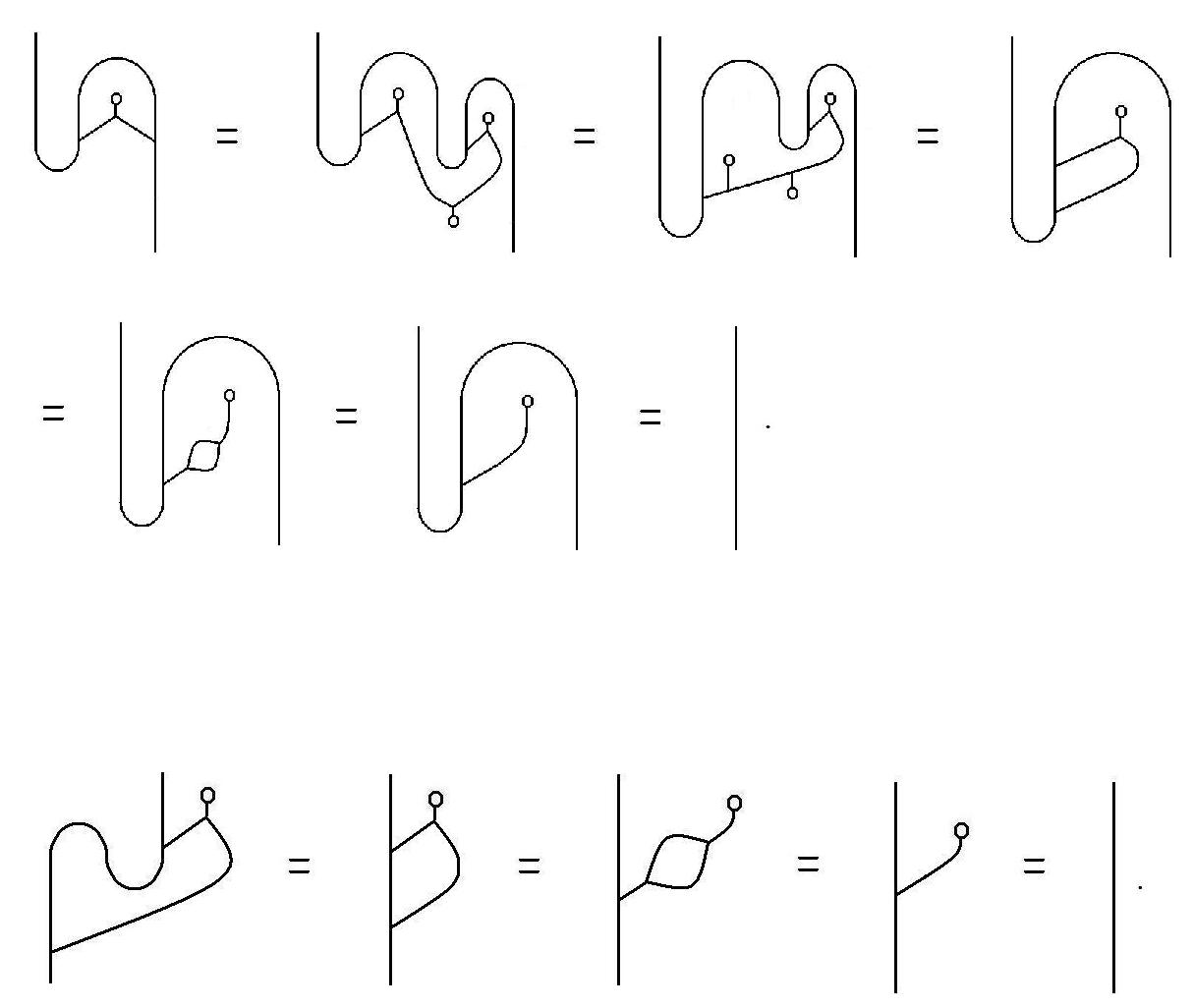}
\end{center}
 
\end{proof}

Combining Theorem \ref{t1} and Proposition \ref{p} we get

\begin{proposition}
A right module $M$ over a separable Frobenius $k$-algebra $R$ is finitely
generated and projective as an $R$-module if it is finitely generated as a
$k$-module.
\end{proposition}
 
Now suppose that $F : \mathcal{C} \to k\operatorname{-Mod}$ is a separable
Frobenius monoidal functor from a monoidal category $\mathcal{C}$ to the
category of $k$-modules with image in the finitely generated projective
$k$-modules. We get a separable Frobenius algebra $R = F(I)$ and
a strong monoidal functor $\tilde{F} : \mathcal{C} \to \mathrm{Mod}(R, R)$.
By Proposition \ref{p}, $\tilde{F}$ has its image in the subcategory of
Cauchy modules. The functor $(-)^\ast$ takes the dual object in
$k\operatorname{-Mod}$. Since $R$ is a separable Frobenius algebra we have an
equivalence

$$\mathrm{Mod}(R, R) \simeq \mathrm{Comod}(R, R)\text{.}$$

\noindent It follows that the forgetful functor $U$ preserves colimits. Hence
the coend $L$ of \eqref{coend} is computed as in $k\operatorname{-Mod}$.  It can
be seen that $L$ is defined by the usual coend formula of the
Tannaka construction (see \cite{JS}) applied to the functor $F : \mathcal{C} \to
k\operatorname{-Mod}$:

$$L = \int^{M}F(M)^\ast\otimes_kF(M)\text{.}$$

\noindent A bialgebroid with a separable Frobenius base
algebra is the same as a weak bialgebra \cite{Sch}. It follows that $L$ is a
weak bialgebra. 

As a formal consequence of Theorem \ref{theorem} we obtain:

\begin{theorem}
Let $k$ be a field. Suppose that $\mathcal{C}$ is a locally finite $k$-linear
abelian monoidal category and $F : \mathcal{C} \to k\operatorname{-Vect}$ is a
faithful, exact, separable Frobenius monoidal functor with image in the finite 
dimensional $k$-vector spaces. Then there exists a weak bialgebra $L$ and a
monoidal equivalence between $\mathcal{C}$ and the monoidal category of comodules over $L$.
\end{theorem}

The monoidal category of comodules over the weak bialgebra $L$ is defined by
regarding $L$ as a bialgebroid over a separable Frobenius algebra.

\end{document}